\documentclass[12pt,a4paper]{amsart}
\usepackage{amsfonts}
\usepackage{amsthm}
\usepackage{amsmath,amssymb,amsfonts,mathtools,enumerate,tikz-cd}
\usepackage{amscd}
\usepackage[latin2]{inputenc}
\usepackage{t1enc}
\usepackage[mathscr]{eucal}
\usepackage{indentfirst}
\usepackage{graphicx}
\usepackage{graphics}
\usepackage{pict2e}
\usepackage{epic}
\numberwithin{equation}{section}
\usepackage[margin=2.9cm]{geometry}
\usepackage{epstopdf} 
\usepackage{hyperref}
\usepackage{multirow}
\usepackage{diagbox}

\theoremstyle{plain}
\newtheorem{thm}{Theorem}[section]
\newtheorem{lem}[thm]{Lemma}

\newtheorem{prop}[thm]{Proposition}

 \theoremstyle{definition}
\newtheorem{defn}[thm]{Definition}

\newcommand{\Hom}{{\rm{Hom}}}

\newcommand{\cl}[1]{\mathcal{#1}}
\newcommand{\bb}[1]{\mathbb{#1}}
\newcommand{\tr}{\operatorname{tr}}
\newcommand{\PSL}{\operatorname{PSL}}
\newcommand{\SL}{\operatorname{SL}}
\newcommand{\Gal}{\operatorname{Gal}}
\newcommand{\hfal}{\operatorname{h_{Fal}}}
\newcommand{\Ind}{\operatorname{Ind}}

\begin{document}

\title{Unitary $\PSL_2$ CM Fields and the Colmez Conjecture}

\author[S. Parenti]{Solly Parenti}

\address{University of Wisconsin-Madison \\ Department of Mathematics \\
Madison WI 53703} 

\email{sparenti@wisc.edu}



\begin{abstract} We study certain unitary CM fields whose Galois closure has Galois group $\PSL_2(\bb{F}_q) \times \bb{Z}/2\bb{Z}$.  After investigating the CM types of these fields, we turn towards Colmez's conjectural formula on the Faltings heights of CM abelian varieties.  We explicitly calculate the class functions appearing in the statement of the conjecture and then apply refinements of the recently proven average version to establish Colmez's conjecture in this case.
\end{abstract}

\maketitle

\section{Introduction} In a 1993 paper \cite{Col}, Pierre Colmez conjectured a relationship between the Faltings height of a CM abelian variety and log derivatives of certain $L$-functions, generalizing the classic Chowla-Selberg Formula.  While several cases and an average version of the conjecture have been proven, it remains for the most part open.

In this paper, we will verify the Colmez conjecture in the following specific case.  Let $F$ be a totally real number field whose Galois closure, $F^c$, has Galois group $\Gal(F^c/\bb{Q}) \cong \PSL_2(\bb{F}_q)$, where $q$ is an odd prime power.  Moreover, we also suppose $F$ is the fixed field by the Borel subgroup.  We let $k$ be any imaginary quadratic field and define a CM field $E:=kF$. 
\begin{thm}
The Colmez conjecture holds for $E$.
\end{thm}
That is, let $\Phi$ be a CM type of $E$ and $X_{\Phi}$ an abelian variety with CM by $\cl{O}_E$, the ring of integers of $E$, and CM type $\Phi$.  Then the Faltings height of $X_{\Phi}$ can be written in terms of log derivatives of $L$-functions coming from $\Phi$, in a manner that will be made more precise in section 2. 

Given a CM type $\Phi$ of $E$, we will calculate which $L$-functions appear in the statement of the conjecture.  Surprisingly, these $L$-functions depend on $\Phi$ in a very simple way.  From this information, we will apply a refinement of the average conjecture due to Yang and Yin \cite{YangYin}.

In Section 2 of this paper, we will provide a more thorough introduction to the Colmez conjecture.  We also refer the reader to the well written introductions in \cite{BSM} and \cite{Yang}.  Section 3 establishes the notation and set up for our results.  In Section 4 we partially classify the CM types of $E$ and Section 5 contains the necessary group and representation theoretic data about $\PSL_2(\bb{F}_q)$. Section 6 contains the main calculation involved in the proof of Theorem 1.1 and Section 7 concludes the proof.

\section{Background}

Let $E$ be a CM number field, that is, a totally imaginary number field that is a degree 2 extension of $F$, a totally real number field.  Every such CM field has a unique automorphism $\rho:E \rightarrow E$ that acts as complex conjugation.  This is the unique automorphism with the property that for every $x \in E$ and every embedding $\sigma:E \hookrightarrow \bb{C}$, we have that $x^{\rho\sigma} = \overline{x^{\sigma}}$. In this paper, we will focus on unitary CM fields.

\begin{defn}A CM field $E$ is called unitary if it contains an imaginary quadratic field, which we will often call $k$. In this case, $E$ is the compositum of $k$ and the totally real number field $F$.
\end{defn}

The embeddings of $E$ into $\bb{C}$ come in conjugate pairs.  A choice of one embedding from every conjugate pair constitutes a CM type.  Viewing a disjoint union as a map to a discrete space leads to an alternate equivalent definition of a CM type.  We will primarily use Definition 2.2, but we include Definition 2.3 as it provides motivation for subsequent constructions in this section.

\begin{defn}
Let $E$ be a CM number field with complex conjugation $\rho$.  A subset $\Phi \subseteq \Hom(E,\bb{C})$ is a CM type if the following two equalities hold
\begin{align*}
\Phi \cap \rho\Phi &= \varnothing ,\\
\Phi \cup \rho\Phi &= \Hom(E,\bb{C}).
\end{align*}
\end{defn}

\begin{defn}
Let $E$ be a CM number field with complex conjugation $\rho$.  A function $\Phi:\Hom(E,\bb{C}) \rightarrow \{0,1\}$ is a CM type if $\Phi(\sigma) + \Phi(\rho\sigma) = 1$ for every $\sigma \in \Hom(E,\bb{C})$.
\end{defn}

Let $E$ be a CM field, $M$ be a CM field containing $E$, and $\Phi$ a CM type of $E$.  We can extend $\Phi$ to a CM type of $M$, $\Phi^M$, via
\[
\Phi^M = \{ \sigma:M \hookrightarrow \bb{C} : \sigma |_{E} \in \Phi\}.
\]
Often, we will take $M= E^c$, the Galois closure of $E$, and we write $\Phi^c$ for the extension of $\Phi$ to $E^c$.

\begin{defn}
Let $E$ be a Galois CM field, and let $\Phi_1$ and $\Phi_2$ be CM types of $E$.  We say $\Phi_1$ is equivalent to $\Phi_2$ if there is some $\tau \in \Gal(E^c/\bb{Q})$ with $\tau \Phi_1 = \Phi_2$.
\end{defn}

The initial data of the Colmez conjecture is a pair $(E,\Phi)$, with $E\subseteq \bb{C}$ a CM field $E$ and $\Phi$ a CM type of $E$.  The important quantities in the conjecture are invariant under extending $E$ to a larger CM field (and using the extended CM type) as well as replacing $\Phi$ by an equivalent CM type.

Let $\Phi^c$ denote the extension of $\Phi$ to a CM type on $E^c$, the Galois closure of $E$.  We identify $\Hom(E^c,\bb{C})$ with $\Gal(E^c/\bb{Q})$.  This choice is only well defined up to conjugacy.  Our construction will produce a class function, $A_{\Phi}^0$ and this choice of identification will not affect $A_{\Phi}^0$.  Write $\Phi^c$ as
\[
\Phi^c = \sum_{\sigma \in \Phi^c} \sigma.
\]
In writing this sum, we are thinking of $\Phi^c$ as an element of the group ring $\bb{C}[\Gal(E^c/\bb{Q})]$ and we will do arithmetic in this ring.  The justification behind this change of viewpoint is the following isomorphism of rings.  Here, $\operatorname{Maps}(\Gal(E^c/\bb{Q}),\bb{C})$ denotes the maps of sets from $\Gal(E^c/\bb{Q})$ to $\bb{C}$.  We give this a ring structure with operations pointwise addition of functions and convolution.
\begin{align*}
\bb{C}[\Gal(E^c/\bb{Q})] &\xrightarrow{\cong} \operatorname{Maps}(\Gal(E^c/\bb{Q}),\bb{C}) \\
\sum n_{\sigma}\sigma &\mapsto 
  \begin{cases} 
   \Gal(E^c/\bb{Q}) & \rightarrow \bb{C} \\
   \quad\quad \sigma       &\mapsto n_{\sigma}
  \end{cases}\\
\end{align*}

Consider the reflex CM type $\widetilde{\Phi^c}$ given by 
\[
\widetilde{\Phi^c} := \sum_{\sigma \in \Phi^c} \sigma^{-1}.
\]
Next, we define $A_{\Phi}$ by taking a convolution of $\Phi^c$ and $\widetilde{\Phi^c}$.  More precisely, 
\[
A_{\Phi}:= \frac{1}{[E^c:\bb{Q}]} \Phi^c \widetilde{\Phi^c}.
\]
Define the class function $A_{\Phi}^0$ by projecting $A_{\Phi}$ onto the space of class functions on $\Gal(E^c/\bb{Q})$.  I.e., 
\[
A_{\Phi}^0 := \frac{1}{[E^c:\bb{Q}]}\sum_{g \in \Gal(E^c/\bb{Q})} gA_{\Phi}g^{-1}.
\]

As the characters of the irreducible representations of $\Gal(E^c/\bb{Q})$ form a basis for the space of class functions, we can write $A_{\Phi}^0 = \sum_{\chi}a_{\chi}\chi$ with $a_{\chi} \in \bb{C}$ as $\chi$ ranges through the irreducible characters of $\Gal(E^c/\bb{Q})$.  For $s \in \bb{C}$, let $L(s,\chi)$ denote the Artin $L$-function of $\chi$ and let $f_{\chi}$ denote the Artin conductor of $\chi$. Finally, define the function $Z(s,A_{\Phi}^0)$ by
\[
Z(s,A_{\Phi}^0):= \sum_{\chi} a_{\chi}Z(s,\chi),\quad \quad Z(s,\chi):= \frac{L'(s,\chi)}{L(s,\chi)} + \frac{1}{2} \log f_{\chi}.
\]

The other side of the Colmez conjecture is more geometric.  As before, let $E$ be a CM field of degree $2n$ over $\bb{Q}$ and let $\Phi$ be a CM type of $E$.  Let $X_{\Phi}$ be an abelian variety with CM by $(\cl{O}_E, \Phi)$, where $\cl{O}_E$ is the ring of integers of $E$.  That is, $\dim(X_{\Phi}) = n$, we have an embedding $\cl{O}_E \hookrightarrow \text{End}(X_{\Phi})$, and $\Phi$ describes the action of $\cl{O}_E$ on the holomorphic differentials.

As $X_{\Phi}$ has complex multiplication, we can find a number field $L$ over which $X_{\Phi}$ is defined and has everywhere good reduction \cite{MilneCM}.  Let $\cl{X}_{\Phi}$ be the Neron model of $X_{\Phi}$ defined over $\cl{O}_L$ and let $\epsilon:\text{Spec}(\cl{O}_L) \rightarrow \cl{X}_{\Phi}$ be the zero section.  Let $\omega_{\cl{X}_{\Phi}/\cl{O}_L} = \epsilon^{\ast}(\Lambda^n \Omega_{\cl{X}_{\Phi}/\cl{O}_L})$ and take a non-zero $\alpha \in\omega_{\cl{X}_{\Phi}/\cl{O}_L}$. There are different normalizations of the Faltings height, and we will follow \cite{YangYin} and \cite{Col}. The stable Faltings height of $X_{\Phi}$ is defined by
\[
\hfal(X_{\Phi}):= \frac{-1}{2[L:\bb{Q}]} \sum_{\sigma:L \hookrightarrow \bb{C}} \log \Bigg\lvert \int_{X_{\Phi}^{\sigma}(\bb{C})} \alpha^{\sigma} \wedge \overline{\alpha^{\sigma}}\Bigg\rvert + \log \lvert \omega_{\cl{X}_{\Phi}/\cl{O}_L} / \cl{O}_L \alpha \rvert.
\]
This is independent of choice of $L$ and $\alpha$. We define the Faltings height of a CM type by 
\[
\hfal(\Phi) = \frac{1}{[E:\bb{Q}]} \hfal(X_{\Phi}).
\]
Under this choice of normalization, $\hfal(\Phi)$ does not change if we extend $\Phi$ to the extended CM type on a larger CM field containing $E$.

Given a CM field $E$ and a CM type $\Phi$, Colmez's conjecture is the following equality:
\begin{equation} 
\hfal(\Phi) = -Z(0,A_{\Phi}^0).
\end{equation}

Known cases include when $E$ is abelian \cite{Col}, \cite{Obus}, when $[E:\bb{Q}] = 4$ \cite{Yang}, \cite{Yang10}, \cite{Yang13}, and when the Galois group is as large as possible, so called Weyl CM fields \cite{BSM}.  

Colmez's conjecture also suggests a formula for the average Faltings height amongst abelian varieties with CM by a given CM field $E$.  For $E$ a CM field of degree $2n$, let $\Phi(E)$ denote the set of all CM types of $E$.  The following formula was also conjectured by Colmez in \cite{Col} and was recently proved independently by 2 groups \cite{AGHM},\cite{YZ}:
\begin{equation}
\frac{1}{2^n} \sum_{\Phi \in \Phi(E)}\hfal(\Phi) = -\frac{1}{2^n}\sum_{\Phi \in \Phi(E)} Z(0,A_{\Phi}^0) = -\frac{1}{4n}Z(0,\chi_{E/F}) -\frac{1}{4} \log(2\pi).
\end{equation}
In the above equation, $F$ denotes the totally real subfield of $E$ and $\chi_{E/F}$ is the quadratic character associated to the extension $E/F$.   

\section{Preliminary Set Up}
In this paper we will consider the following set up.  Let $k \subseteq \bb{C}$ be an imaginary quadratic field.  Let $q$ be an odd prime power, and let $F$ be a totally real number field of degree $q+1$ such that, if $F^c$ is the Galois closure we have that 
\[
\Gal(F^c/\bb{Q}) \cong \PSL_2(\bb{F}_q), \quad \quad \Gal(F^c/F) \cong B,
\]
where $B$ is the Borel subgroup in $\PSL_2(\bb{F}_q)$,
\[
B:= \Bigg\{ \begin{bmatrix} a & b \\ 0 & a^{-1} \end{bmatrix} : a\in \bb{F}_q^{\ast}, b \in \bb{F}_q\Bigg\}.
\]
Then $E:=kF$ is a unitary CM field of degree $2(q+1)$ and $E^c = kF^c$ is the Galois closure of $E$ and is also a unitary CM field with $\Gal(E^c/\bb{Q}) \cong \PSL_2(\bb{F}_q) \times \bb{Z}/2\bb{Z}$.

The embeddings $F \hookrightarrow \bb{C}$ can be identified with coset representatives of $\PSL_2(\bb{F}_q)/B$ and we will take the following representatives:
\[
w:= \begin{bmatrix} 0 & 1 \\ -1 & 0 \end{bmatrix} \quad \quad \Bigg\{ n_{-}(i) := \begin{bmatrix}  1 & 0 \\ i & 1 \end{bmatrix}\Bigg\}_{i \in \bb{F}_q}.
\]

Another subgroup that will show up briefly is $U\subseteq B$, the unipotent upper triangular matrices,
\[
U:= \Bigg\{ \begin{bmatrix} 1 & b \\ 0 & 1 \end{bmatrix} :  b \in \bb{F}_q\Bigg\}.
\]

The main theorem of this paper is that the Colmez Conjecture holds in this setting.
\begin{thm}
Let $\Phi$ be a CM type of $E$, and let $\chi_{k/\bb{Q}},\chi_{E/F}$ denote the quadratic characters corresponding to the quadratic extensions $k/\bb{Q}$ and $E/F$ respectively.  Then, 
\begin{align*}
\hfal(\Phi) &= -Z(0,A_{\Phi}^0) \\
&= -\frac{1}{4}Z(0,\zeta_k)+\frac{\epsilon(q+1-\epsilon)}{q(q+1)}Z(0,\chi_{k/\bb{Q}}) - \frac{\epsilon(q+1-\epsilon)}{q(q+1)^2}Z(0,\chi_{E/F}),
\end{align*}
where $\epsilon$ is an integer with $0\leq\epsilon\leq q+1$, related to the signature of $\Phi$, a notion to be defined in the following section.
\end{thm}


\section{Classification of CM Types}
A CM type of $E$ consists of $q+1$ embeddings of $E$ into $\bb{C}$ such that no two of the embeddings are conjugate.  An embedding $E\hookrightarrow\bb{C}$ is uniquely determined by a pair of embeddings $F \hookrightarrow \bb{C}$ and $k \hookrightarrow\bb{C}$.  Therefore the data of a CM type consists of a choice of one of the embeddings of $k \hookrightarrow \bb{C}$ for each $F \hookrightarrow \bb{C}$.  
\begin{defn}
A CM type $\Phi$ of $E$ has signature $(q+1-\epsilon,\epsilon)$ if $q+1-\epsilon$ of the embeddings in $\Phi$ give rise to the identity $k \hookrightarrow \bb{C}$ and the remaining $\epsilon$ of the embeddings in $\Phi$ give rise to the conjugate $k \hookrightarrow\bb{C}$.
\end{defn}
We use the coset representatives $\{w,n_{-}(i)\}_{i \in \bb{F}_q}$ to specify embeddings of $F$ into $\bb{C}$ and $\{1,\rho\}$ for the embeddings of $k$ into $\bb{C}$.  An arbitrary CM type $\Phi$ of $E$ is of the form
\[
\Phi =\rho^{\epsilon_{\infty}} w + \sum_{i \in \bb{F}_q} \rho^{\epsilon_{i}}n_{-}(i),
\]
where $\epsilon_i$ is either 0 or 1. The signature of such a $\Phi$ is $(q+1-\epsilon,\epsilon)$ where $\epsilon = \sum \epsilon_{i}$.

Recall that we are using the isomorphism $\bb{C}[\Gal(E^c/\bb{Q})] \cong \operatorname{Maps}(\Gal(E^c/\bb{Q}),\bb{C})$, as well as identifying $\Gal(E^c/\bb{Q})$ with $\Hom(E^c,\bb{C})$.  Under these correspondences, the element $\rho^{\epsilon_{\infty}}w \in \bb{C}[\Gal(E^c/\bb{Q})]$ represents the embedding of $E$ into $\bb{C}$ which is $w$ on $F$ and the identity on $k$ if $\epsilon_{\infty} = 0$ or the conjugate embedding on $k$ if $\epsilon_{\infty} = 1$.  If $\epsilon_{\infty} = 0$, we will simply write this element as $w$. 

The following lemma provides a limited classification of CM types of $E$.  Parts $(a)$ and $(b)$ of Lemma 4.2 apply to any unitary CM field, while parts $(c)-(e)$ only apply to the specific CM fields considered in this paper.

\begin{lem}
\begin{enumerate}[(a)]
\item There is a bijection between equivalence classes of CM types of signature $(q+1-\epsilon,\epsilon)$ and equivalence classes of CM types of signature $(\epsilon, q+1-\epsilon)$ for any $0 \leq \epsilon \leq q+1$.
\item There is one equivalence class of CM types of signature $(q,1)$.
\item There is one equivalence class of CM types of signature $(q-1,2)$.
\item If $q \equiv 3 \pmod{4}$, there is one equivalence class of CM types of signature $(q-2,3)$.
\item If $q \equiv 1 \pmod{4}$, there are two equivalence classes of CM types of signature $(q-2,3)$.
\end{enumerate}
\end{lem}
\begin{proof}
Given a CM type $\Phi$ as above, we can extend $\Phi$ to $\Phi^c$ a CM type on $E^c$.  Since $E$ is the fixed field by $B$, 
\[
\Phi^c= \rho^{\epsilon_{\infty}} wB + \sum_{i \in \bb{F}_q} \rho^{\epsilon_{i}}n_{-}(i)B.
\]

Two CM types $\Phi_1$ and $\Phi_2$ of $E$ are equivalent if there is some $g \in \text{Gal}(E^c/\bb{Q})$ such that $g \Phi_1^c = \Phi_2^c$.

Part $(a)$ follows from multiplication by $\rho$.

The subsequent parts of this lemma involve classifying CM types within a given signature $(q+1-\epsilon,\epsilon)$.  This amounts to studying the orbit of $\epsilon$ many cosets in $\PSL_2(\bb{F}_q)/B$ under the action of left multiplication by $\PSL_2(\bb{F}_q)$.

For $(b)$, note that $\PSL_2(\bb{F}_q)$ acts transitively on the elements of $\PSL_2(\bb{F}_q)/B$.

To classify the orbit of 2 element subsets of $\PSL_2(\bb{F}_q)/B$ for part $(c)$, first take some $\{g_1B,g_2B\}$.  By part $(b)$, we can assume that $g_1B = wB$.  Then, there is some $i \in \bb{F}_q$ such that $g_2B = n_{-}(i)B$ and thus $n_{-}(-i) \{wB,n_{-}(i)B\} = \{wB,B\}$.  

Given 3 cosets $\{g_1B,g_2B,g_3B\}$, by part $(c)$ we may assume $g_1B = wB$ and $g_2B = B$.  In that case, $g_3B = n_{-}(i)B$ for some $i \in \bb{F}_q^{\ast}$.  For $x \in \bb{F}_q^{\ast}$, we have
\[
\begin{bmatrix}x^{-1} & 0 \\ 0 & x \end{bmatrix} \{wB,B, n_{-}(i)B\} = \{wB,B,n_{-}(i x^2)B\}.
\]
From this, we see that there are at most 2 equivalence classes of CM types of signature $(q-2,3)$, based on whether or not $i$ is a square.  Representatives for the equivalence classes are given by the following, where $\Delta \in \bb{F}_q^{\ast}$ is a non-square,
\[
\{wB,B,n_{-}(1)B\}, \quad \quad \{wB,B,n_{-}(\Delta)B\}.
\]

If $q \equiv 3 \pmod{4}$, then $-1$ is not a square in $\bb{F}_q$, and the following calculation shows that all of the CM types of signature $(q-2,3)$ are equivalent:
\[
n_{-}(-1) \{wB,B,n_{-}(1)B\} = \{wB, n_{-}(-1)B,B\}.
\] 

Finally for part $(e)$, we need to show that the two CM types are distinct.  A quick calculation shows that, for $q \equiv 1 \pmod{4}$,  the stabilizer of $\{wB,B,n_{-}(1)B\}$ is the following subgroup (note as we are in $\PSL_2(\bb{F}_q)$, this subgroup is independent of choice of square root of $-1$)
\[
\Bigg\{ \begin{bmatrix} 1 & 0 \\ 0 & 1 \end{bmatrix},  \begin{bmatrix} 0 & 1 \\ -1 & 1 \end{bmatrix},  \begin{bmatrix} -1 & 1\\ -1 & 0 \end{bmatrix},  \begin{bmatrix} 0 & \sqrt{-1} \\ \sqrt{-1}  & 0 \end{bmatrix},  \begin{bmatrix} \sqrt{-1}  & 0 \\ \sqrt{-1} & -\sqrt{-1}  \end{bmatrix},  \begin{bmatrix} \sqrt{-1}  & -\sqrt{-1}  \\ 0 & -\sqrt{-1}  \end{bmatrix}\Bigg\}.
\]
\end{proof}

Unfortunately there does not appear to be any further clear patterns regarding the number of CM types of $E$.  The action of $\PSL_2(\bb{F}_q)$ on $\PSL_2(\bb{F}_q)/B$ is the same as the action of $\PSL_2(\bb{F}_q)$ on $\bb{P}^1(\bb{F}_q)$ and we can compute the number of orbits of this action in Magma. The following table shows the number of CM types of a given signature for a number of small prime powers.  From the table, it appears that the number of inequivalent CM types grows rapidly.  

In Section 6 of this paper, we will explicitly compute $A_{\Phi}^0$ for a CM type $\Phi$.  Despite the large number of inequivalent CM types of a given signature, Theorem 6.1 will show that $A_{\Phi}^0$ depends only on the signature of $\Phi$.

\begin{center}
Number of Equivalence classes of CM types of signature $(q+1-\epsilon,\epsilon)$
\end{center}

\begin{center}
\begin{tabular}{| l | l | l | l |l |l |l |l |l |}
    \hline
     \diagbox{$q$}{$\epsilon$}  &1 & 2 & 3 & 4 & 5 & 6 & 7 \\ \hline
     7 &    1& 1 & 1 & 3 & 1 & 1 &1 \\ \hline
    9 &   1 & 1 & 2 & 3 & 4 & 3 & 2 \\ \hline
    11 & 1 & 1 & 1 & 2 & 2 & 6 & 2 \\ \hline
    13 & 1 & 1 & 2 & 4 & 5 & 7 & 10 \\ \hline
    17 & 1 & 1 & 2 & 4 & 8 & 15 & 20 \\ \hline
    19 & 1 & 1 & 1 & 5 & 6 & 19 & 26 \\ \hline
    23 & 1 & 1 & 1 & 5 & 7 & 34 & 57 \\ \hline
    25 & 1 & 1 & 2 & 7 & 16 & 45 & 108 \\ \hline
    27 & 1 & 1 & 1 & 6 & 10 & 54 & 124 \\ \hline
    29  & 1 & 1 & 2 & 6 & 19 & 68 & 194\\ \hline
    31 & 1 & 1 & 1 & 8 & 15 & 83 & 233 \\ \hline
   
    \end{tabular}

\end{center}

\section{Character Theory of $\PSL_2(\bb{F}_q)$}
There are two ways to classify conjugacy classes in $\PSL_2(\bb{F}_q)$ and we will use both.  First, we will classify conjugacy classes in $\SL_2(\bb{F}_q)$. Fix $\Delta \in \bb{F}_q^{\ast}$ a non-square.
\begin{prop}
Let $A,B \in \SL_2(\bb{F}_q)$ such that $\text{tr}(A) = \text{tr}(B) \neq \pm 2$, where $\text{tr}$ denotes the trace of a matrix.  Then $A$ and $B$ are conjugate.
\end{prop}
\begin{proof}
Since $A$ and $B$ have the same trace, determinant 1, and are $2 \times 2$ matrices, they have the same characteristic polynomial.  This polynomial cannot have repeated roots, as the trace is not equal to $\pm 2$.  Call this polynomial $f(T)$ and write $f(T) = (T-\lambda_1)(T-\lambda_2)$ with $\lambda_1,\lambda_2 \in \overline{\bb{F}_q}$ and $\lambda_1 \neq \lambda_2$.
Since $\lambda_1$ and $\lambda_2$ are distinct, both $A$ and $B$ are conjugate in $\SL_2(\overline{\bb{F}_q})$ to $\begin{bmatrix} \lambda_1 & 0 \\ 0 & \lambda_2 \end{bmatrix}$.  Two matrices in $\SL_2(\bb{F}_q)$ which are conjugate in $\SL_2(\overline{\bb{F}_q})$ must be conjugate in $\SL_2(\bb{F}_q)$.

\end{proof}

For matrices of trace $\pm 2$, we use the following two propositions.
\begin{prop}
Let $A = \begin{bmatrix} a & b \\ c& d \end{bmatrix} \in \SL_2(\bb{F}_q)$ with $\text{tr}(A) = \delta2$ with $\delta \in \{\pm 1\}$.
\begin{enumerate}[(a)]
\item If $b = 0$, then $A$ is conjugate to $\begin{bmatrix} \delta & -c \\ 0 & \delta \end{bmatrix}$.
\item If $b \neq 0$, then $A$ is conjugate to $\begin{bmatrix} \delta & b \\ 0 & \delta \end{bmatrix}$.
\end{enumerate}
\end{prop}
\begin{proof}
If $b = 0$, then we have $ad = 1$ and $a+d = \delta 2$, and so $a = d = \delta$.  Then,
\[
\begin{bmatrix} 0 & 1 \\ -1 & 0\end{bmatrix} \begin{bmatrix} \delta & 0 \\ c & \delta \end{bmatrix} \begin{bmatrix} 0 & -1 \\ 1 & 0 \end{bmatrix} = \begin{bmatrix}\delta & -c \\ 0 & \delta \end{bmatrix}.
\]

If $b \neq 0$, then
\[
\begin{bmatrix} 1 & 0 \\ \frac{a-d}{2b} & 1\end{bmatrix} \begin{bmatrix} a & b \\ c & d \end{bmatrix} \begin{bmatrix} 1 & 0 \\ -\frac{a-d}{2b} & 1 \end{bmatrix} = \begin{bmatrix}\delta & b \\ 0 & \delta \end{bmatrix}.
\]
\end{proof}
\begin{prop}
For $\delta \in \{\pm 1\}$ and $x,y \in \bb{F}_q$, the two matrices $\begin{bmatrix} \delta & x \\ 0 & \delta \end{bmatrix}$ and $\begin{bmatrix} \delta & y \\ 0& \delta\end{bmatrix}$ are conjugate (in $\SL_2(\bb{F}_q)$) if and only if either of the following occurs:
\begin{enumerate}[(a)]
\item $x=y=0$, or
\item $x,y \in \bb{F}_q^{\ast}$ and $x \equiv y$ in $\bb{F}_q^{\ast}/(\bb{F}_q^{\ast})^2$.
\end{enumerate}
\end{prop}
\begin{proof}
The case $x = y =0$ is clear, so assume $x,y \in \bb{F}_q^{\ast}$.  
An arbitrary conjugate of $\begin{bmatrix} \delta & x \\ 0 & \delta \end{bmatrix}$ is of the following form:
\[
\begin{bmatrix} a & b \\ c & d\end{bmatrix} \begin{bmatrix} \delta & x \\ 0 & \delta \end{bmatrix} \begin{bmatrix} d & -b \\ -c & a \end{bmatrix} = \begin{bmatrix}\delta-acx & a^2x \\ -c^2x & \delta+acx \end{bmatrix}.
\]
For this to be equal to $\begin{bmatrix} \delta & y \\ 0 & \delta \end{bmatrix}$, we must have $a^2x = y$.
\end{proof}

In $\PSL_2(\bb{F}_q)$, trace is only well defined as an element of the set $\bb{F}_q/\{\pm\}$.  Combining the above ideas leads to the following classification of conjugacy classes in $\PSL_2(\bb{F}_q)$.
\begin{prop}
The conjugacy classes of $\PSL_2(\bb{F}_q)$ are as follows.
\begin{itemize}
\item $\operatorname{Id}:$ This is the identity matrix.
\item $\tr^{\square}(2):$ These are the matrices conjugate to $\begin{bmatrix} 1 & a^2 \\ 0 & 1 \end{bmatrix} $ for any $a \in \bb{F}_q^{\ast}$.
\item $\tr^{\not\square}(2):$ If $\Delta$ is a fixed non-square in $\bb{F}_q^{\ast}$, then these are the matrices conjugate to $\begin{bmatrix} 1 & \Delta a^2 \\ 0 & 1 \end{bmatrix} $ for any $a  \in \bb{F}_q^{\ast}$. 
\item $\tr(x):$ For $x \in \bb{F}_q\setminus\{\pm 2\}$, these are the matrices of trace equal to $\pm x$.
\end{itemize}
\end{prop}

Following \cite{Adams}, we can also think of these conjugacy classes slightly differently and we will occasionally switch between the two viewpoints.  In particular, we will classify the matrices of trace $\pm 2$ exactly the same, but we will split up the other matrices slightly differently.
\begin{prop}
The conjugacy classes of $\PSL_2(\bb{F}_q)$ are as follows.
\begin{itemize}
\item $\operatorname{Id}$ 
\item $\tr^{\square}(2)$
\item $\tr^{\not\square}(2)$
\item $\tr(0)$
\item $\tr(x+x^{-1}):$ These are the matrices conjugate to $\begin{bmatrix} x & 0 \\ 0 & x^{-1} \end{bmatrix}$ for some $x \in \bb{F}_q^{\ast}$ such that $x \neq \pm 1$, and if $q \equiv 1 \pmod{4}$ we also require $x \neq \pm \sqrt{-1}$.
\item $\tr(2x):$ These are the matrices conjugate to $\begin{bmatrix} x & \Delta y \\ y & x \end{bmatrix}$ where $z:=x+ \sqrt{\Delta} y$ is a norm 1 element of $\bb{F}_{q^2}^{\ast}$ such that $z \neq \pm 1$, and if $q \equiv 3 \pmod{4}$ we also require $z \neq \pm\sqrt{\Delta}$.
\end{itemize}
\end{prop}

Aside from a few exceptions, most of the conjugacy classes are related to the split torus, $B$, or the non-split torus, the norm 1 elements of $\bb{F}_{q^2}^{\ast}$ considered as a subgroup of $\PSL_2(\bb{F}_q)$.  Based on whether $q \equiv 1 \text{ or } 3 \pmod{4}$, the trace 0 conjugacy class would technically fall into one of the above classifications, but this conjugacy class behaves differently and we consider it separately.  The irreducible characters follow a similar pattern. Most irreducible characters come from characters of the split torus or the non-split torus, with a few exceptions based on whether $q \equiv 1 \text{ or } 3 \pmod{4}$.

\begin{prop}[\cite{Adams}] The following is a complete list of the irreducible characters of $\PSL_2(\bb{F}_q)$.
\begin{itemize}
\item $\chi_0:$ The trivial character.
\item $\chi_1:$ The induced representation $\Ind_{B}^{\PSL_2(\bb{F}_q)}(\chi_0)$ decomposes into a direct sum of the trivial representation and an irreducible $q$ dimensional representation with character $\chi_1$.
\item $\{\chi_{\alpha}\}_{\alpha}:$ Let $\alpha$ be a non-quadratic character of $\bb{F}_q^{\ast}/\{\pm\} \cong B/U$ and consider $\alpha:B \rightarrow \bb{C}$.  Then $\chi_{\alpha}$ is the character of the induced representation $\Ind_B^{\PSL_2(\bb{F}_q)}(\alpha)$.  We have that $\chi_{\alpha} = \chi_{\beta}$ if and only if $\beta = \alpha^{-1}$.
\item $\{\pi(\eta)\}_{\eta}:$ Let $\bb{E}_1$ denote the norm 1 elements of $\bb{F}_{q^2}^{\ast}/\{\pm\}$ considered as a subgroup of $\PSL_2(\bb{F}_q)$. Let $\eta$ be a non-quadratic character of $\bb{E}_1$.  Then, $\pi(\eta)$ is the character of the cuspidal representation. We have that $\pi(\eta) = \pi(\theta)$ if and only if $\theta = \eta^{-1}$.
\item $\{\omega_e^+,\omega_e^-\}:$ The oscillator representations for $q \equiv 1 \pmod{4}$.  Let $\alpha_0$ be the non-trivial quadratic character of $B/U$.  Then $\Ind_B^{\PSL_2(\bb{F}_q)}(\alpha_0)$ decomposes into two irreducible representations, whose characters are given by $\omega_e^{+}$ and $\omega_e^{-}$.
\item $\{\omega_o^+,\omega_o^-\}:$ The oscillator representations for $q \equiv 3 \pmod{4}$. Let $\eta_0$ be the non-trivial quadratic character of $\bb{E}_1$.  Then, the cuspidal representation associated to $\eta_0$ decomposes into two irreducible representations, whose characters are given by $\omega_o^+$ and $\omega_o^-$.

\end{itemize}

\end{prop}

Reproduced here is the character table for $\PSL_2(\bb{F}_q)$ from \cite{Adams} with a slight change in notation to match the notation of this paper.  

A few notes on these tables.  First suppose $q \equiv 3 \pmod{4}$.  Let $\eta_0$ be the unique non-trivial quadratic character of $\bb{E}_1/\{\pm\}$, For the conjugacy class of norm 1 elements $\begin{bmatrix} x & \Delta y \\ y & x\end{bmatrix}$, we write $z = x+\sqrt{\Delta}y$ with $z\in \bb{E}_1/\{\pm\}$, $z \neq \pm 1, \pm \Delta$ and we only take one of $z$ or $z^{-1}$.  For the conjugacy class of matrices with eigenvalues defined over $\bb{F}_q$, we write $\begin{bmatrix} x & 0 \\ 0 & x^{-1} \end{bmatrix}$ with $x \in \bb{F}_q^{\ast}/\{\pm\}$, $x \neq \pm 1$ and we only take one of $x$ or $x^{-1}$.

 When we write $\alpha$, we range over all non-quadratic characters of $\bb{F}_q^{\ast}/\{\pm\}$ but we only take one of $\{\alpha, \alpha^{-1}\}$.  When we write $\eta$, we range over all non-quadratic characters of $\bb{E}_1/\{\pm\}$ where we only take one of $\{ \eta,\eta^{-1}\}$.  We also list the size of each conjugacy class.
 \begin{center}
Character Table of $\PSL_2(\bb{F}_q)$ where $q \equiv 3 \pmod{4}$
\def\arraystretch{1.5}
\begin{tabular}{| c | c | c | c | c | c | c | c |}
\hline
 &$\operatorname{Id}$ & $\begin{bmatrix} 1 & 1 \\ 0 & 1\end{bmatrix}$ &$\begin{bmatrix} 1 & \Delta \\ 0 & 1\end{bmatrix}$ &$\begin{bmatrix} x & 0 \\ 0 & x^{-1} \end{bmatrix}$ & $\begin{bmatrix} x & \Delta y \\ y & x\end{bmatrix}$ &$\tr0$ \\[1ex] 
  \hline
 Size & 1 & $\frac{q^2-1}{2}$ & $\frac{q^2-1}{2}$ & $q(q+1)$ & $q(q-1)$ & $\frac{q(q-1)}{2}$ \\[1ex]
  \hline
  $\chi_0$  & 1 & 1 & 1 & 1 & 1 & 1\\[1ex]
  \hline
  $\chi_1$ & $q$ & 0 & 0  & 1 & -1 & 1\\[1ex]
  \hline
  $\chi_{\alpha}$ & $q+1$& 1 & 1 & $\alpha(x) + \alpha(x^{-1})$ & 0 & 0 \\[1ex]
  \hline
$\pi(\eta)$ & $q-1$ & $-1$ & $-1$ & 0 & $-\eta(z) - \eta(z^{-1})$ & $-2\eta(\sqrt{\Delta}) $\\[1ex]
\hline
$\omega_o^{\pm}$ & $\frac{q-1}{2}$ & $\frac{1}{2}(-1 \pm \sqrt{-q})$ & $\frac{1}{2}(-1 \mp \sqrt{-q}) $& $0$ & $-\eta_0(z)$& $-\eta_0(\sqrt{\Delta})$\\[1ex]
\hline
\end{tabular}
\end{center}

Here are the tables for $q \equiv 1 \pmod{4}$.  We use similar conventions as the previous case, with slight differences.  In looking at matrices of the form $\begin{bmatrix} x & 0 \\ 0 & x^{-1} \end{bmatrix}$, we exclude $x = \pm 1$ and $x = \pm \sqrt{-1}$. Also, $\alpha_0$ is the quadratic character of $\bb{F}_q^{\ast}/\{\pm\}$.

\begin{center}
Character Table of $\PSL_2(\bb{F}_q)$ where $q \equiv 1 \pmod{4}$
\def\arraystretch{1.5}
\begin{tabular}{| c | c | c | c | c | c | c | c |}
\hline
 &$\operatorname{Id}$ & $\begin{bmatrix} 1 & 1 \\ 0 & 1\end{bmatrix}$ &$\begin{bmatrix} 1 & -1 \\ 0 & 1\end{bmatrix}$ &$\begin{bmatrix} x & 0 \\ 0 & x^{-1} \end{bmatrix}$ & $\begin{bmatrix} x & \Delta y \\ y & x\end{bmatrix}$ &$\tr0$ \\[1ex] 
  \hline
 Size & 1 & $\frac{q^2-1}{2}$ & $\frac{q^2-1}{2}$ & $q(q+1)$ & $q(q-1)$ & $\frac{q(q+1)}{2}$ \\[1ex]
  \hline
  $\chi_0$  & 1 & 1 & 1 & 1 & 1 & 1\\[1ex]
  \hline
  $\chi_1$ & $q$ & 0 & 0  & 1 & -1 & 1\\[1ex]
  \hline
  $\chi_{\alpha}$ & $q+1$& 1 & 1 & $\alpha(x) + \alpha(x^{-1})$ & 0 & $2\alpha(\sqrt{-1})$ \\[1ex]
  \hline
$\pi(\eta)$ & $q-1$ & $-1$ & $-1$ & 0 & $-\eta(z) - \eta(z^{-1})$ & $0 $\\[1ex]
\hline
$\omega_e^{\pm}$ & $\frac{q+1}{2}$ & $\frac{1}{2}(1 \pm \sqrt{q})$ & $\frac{1}{2}(1 \mp \sqrt{q})$& $\alpha_0(x)$ & $0$& $\alpha_0(\sqrt{-1})$\\[1ex]
\hline
\end{tabular}
\end{center}

\section{Calculations}
This section is devoted to the proof of Theorem 6.1, the main step in our proof of this specific case of the Colmez conjecture.
\begin{thm}
Let $\Phi$ be a CM type of $E$ of signature $(q+1-\epsilon,\epsilon)$.  Then,
\begin{align*}
A_{\Phi}^0 &= \frac{1}{2}\tr_{E^c/k} -\frac{\epsilon(q+1-\epsilon)}{q(q+1)}(1-\rho)\tr_{E^c/k} + \frac{\epsilon(q+1-\epsilon)}{q(q+1)^2}(1-\rho)\chi_{\Ind_B^{\PSL_2(\bb{F}_q)}(\chi_0)} 
\end{align*}\end{thm}
\begin{proof}[Note] We saw there were many equivalence classes of CM types, even among a fixed signature.  However, $A_{\Phi}^0$ depends only on the signature of $\Phi$.  This observation, along with the result of Yang and Yin showing the Colmez conjecture is true if we average amongst CM types of a given signature \cite{YangYin}, will imply our final result.
\end{proof}
\begin{proof}
Let $\Phi$ be a CM type of $E$ of signature $(q+1-\epsilon,\epsilon)$. Without loss of generality (possibly by replacing $\Phi$ with $\rho\Phi$) we can assume $\Phi$ is of the following form:
\[
\Phi = w + \sum_{i \in \bb{F}_q} \rho^{\epsilon_{i}}n_{-}(i).
\] 
where $\epsilon_{i}$ is either 0 or 1, and $\epsilon = \sum \epsilon_{i}$.

A CM type of $E$ consists of $q+1$ embeddings of $E$ into $\bb{C}$.  Recall that an embedding $E \hookrightarrow \bb{C}$ is the same as a pair of embeddings $F \hookrightarrow \bb{C}$ and $k \hookrightarrow \bb{C}$.  So when we write the element $w$, we interpret that to be the embedding of $E$ into $\bb{C}$ corresponding to $w:F \hookrightarrow\bb{C}$ and the identity $k \hookrightarrow \bb{C}$. 

Next, we need to extend $\Phi$ to $\Phi^c$, the corresponding CM type on the Galois closure, $E^c$, of $E$.  Recall that $\Phi^c$ consists of all the embeddings $E^c \hookrightarrow \bb{C}$ which restrict to the embeddings in $\Phi$.  Since $E$ is the fixed field of $B$, we have that
\begin{align*}
\Phi^c &= wB + \sum_{i \in \bb{F}_q} \rho^{\epsilon_{i}}n_{-}(i)B \\
&= \tr_{E^c/k} + (\rho-1) \sum_{i \in \bb{F}_q} \epsilon_{i}n_{-}(i)B.
\end{align*}

Next, we find $\widetilde{\Phi^c}$ by inverting every element of $\Phi^c$,
\[
\widetilde{\Phi^c} = \tr_{E^c/k} + (\rho-1)\sum_{j \in \bb{F}_q}\epsilon_{j}Bn_{-}(-j).
\]
Since $A_{\Phi} = \frac{1}{[E^c:\bb{Q}]} \Phi^c \widetilde{\Phi^c}$, we have that 
\[
A_{\Phi} = \frac{1}{2}\tr_{E^c/k} + \frac{\epsilon}{q+1}(\rho-1)\tr_{E^c/k} + \frac{1}{q+1}(1-\rho)\sum_{i,j \in \bb{F}_q}\epsilon_{i}\epsilon_{j}n_{-}(i)Bn_{-}(-j).
\]

Finally we need to calculate $A_{\Phi}^0$, where we project $A_{\Phi}$ onto the space of class functions.  The two terms in $A_{\Phi}$ involving $\tr_{E^c/k}$ are already conjugacy invariant and remain unchanged when we pass to $A_{\Phi}^0$ and therefore
\[
A_{\Phi}^0 = \frac{1}{2}\tr_{E^c/k} + \frac{\epsilon}{q+1}(\rho-1)\tr_{E^c/k} + \star,
\]
where in the above equation $\star$ is given by
\[
\star = \frac{1}{[E^c:\bb{Q}]} \sum_{\sigma \in \Gal(E^c/\bb{Q})} \sigma \Bigg( \frac{1}{q+1}(1-\rho)\sum_{i,j \in \bb{F}_q} \epsilon_i\epsilon_j n_{-}(i)Bn_{-}(-j)\Bigg)\sigma^{-1}.
\]
Since $\Gal(E^c/\bb{Q}) \cong \Gal(E^c/k) \times \langle \rho \rangle \cong \PSL_2(\bb{F}_q) \times \bb{Z}/2\bb{Z}$, $\star$ is equal to the following:
\[
\star = \frac{1-\rho}{\# \PSL_2(\bb{F}_q) \cdot (q+1)} \sum_{\sigma \in \PSL_2(\bb{F}_q)} \sigma \Bigg( \sum_{i,j \in \bb{F}_q} \epsilon_i \epsilon_jn_{-}(i)Bn_{-}(-j) \Bigg)\sigma^{-1}.
\]

The proof of the theorem amounts to computing $\star$.  As we sum over conjugates of elements in $\PSL_2(\bb{F}_q)$, we replace each element with its conjugacy class.  That is to say, for $g \in \PSL_2(\bb{F}_q)$, let $C(g)$ denote the conjugacy class of $g$.  Then,
\[
\frac{1}{\#\PSL_2(\bb{F}_q)}\sum_{\sigma \in \PSL_2(\bb{F}_q)} \sigma g \sigma^{-1} = \frac{1}{\#C(g)}C(g).
\]
Therefore, to compute $\star$, we will need to determine the conjugacy class of each element appearing in $\star$.  On the one hand, there are $\epsilon$ many terms with $i = j$.  Since $n_{-}(i)Bn_{-}(-i)$ is conjugate to $B$, when we sum over the conjugates of $n_{-}(i)Bn_{-}(-i)$, it is the same as summing over all of the conjugates of $B$.

For the following parts of this section, we will write out $B$ more explicitly.  Fix $A \subset \bb{F}_q^{\ast}$ a set of representatives for $\bb{F}_q^{\ast}/\{\pm\}$.  That is, $A$ consists of $\frac{q-1}{2}$ elements of $\bb{F}_q^{\ast}$ such that $A/\{\pm\} = \bb{F}_q^{\ast}/\{\pm\}$.  If $q \equiv 3 \pmod{4}$, we take $A$ to consist of all of the squares in $\bb{F}_q^{\ast}$.  In either case, assume $1 \in A$ (and thus $-1 \not\in A$).  The choice of $A$ will not affect the end result, it only serves to make the exposition cleaner.  Then for example,
\[
B = \sum_{(a,b) \in A \times \bb{F}_q}\begin{bmatrix} a & b \\0 & a^{-1} \end{bmatrix}.
\]

There are $\epsilon(\epsilon-1)$ many terms of $\star$ with $i \neq j$ and Proposition 6.2 concerns the conjugacy classes of elements in $n_{-}(i)Bn_{-}(-j)$ for $i \neq j$.

\begin{prop}
Fix $i,j \in \bb{F}_q$ with $i \neq j$.  Then, $n_{-}(i)Bn_{-}(-j)$ consists of $\frac{q-1}{2}$ elements in $\tr( 0)$, $q-1$ elements in $\tr(\alpha)$ for every $\alpha \in \bb{F}_q^{\ast}/\{\pm\}$ with $\alpha \neq \pm 2$, $\frac{q-1}{2}$ elements in $\tr^{\square}(2)$, and $\frac{q-1}{2}$ elements in $\tr^{\not\square}(2)$.

In particular, the conjugacy classes of elements in $n_{-}(i)Bn_{-}(-j)$ are independent of $i$ and $j$. 
\end{prop}
\begin{proof} The elements of $n_{-}(i)Bn_{-}(-j)$ are of the form
\[
n_{-}(i)Bn_{-}(-j) = \sum_{a \in A}\sum_{b \in \bb{F}_q} \begin{bmatrix} a-bj & b \\ ai-bij-a^{-1}j & bi+a^{-1}\end{bmatrix}.
\]

In doing this calculation, we will think of these matrices as elements of $\SL_2(\bb{F}_q)$ and then project down to $\PSL_2(\bb{F}_q)$.  For a fixed $a \in A$, the trace of this matrix is $b(i-j) + a + a^{-1}$, which is affine linear in $b$.  As $b$ ranges through $\bb{F}_q$, the trace takes on every value in $\bb{F}_q$.  When we project down to $\PSL_2(\bb{F}_q)$, this gives us $1$ matrix in $\tr(0)$, 2 matrices of $\tr(\alpha)$ for each $\alpha \in \bb{F}_q^{\ast}/\{\pm\}$ with $\alpha \neq \pm 2$, and two matrices of trace $\pm 2$.  

First, suppose that $q \equiv 3 \pmod{4}$.  We claim that the two matrices of trace $\pm 2$ always consists of 1 element in $\tr^{\square}(2)$ and 1 element in $\tr^{\not\square}(2)$ (Note that since $i \neq j$, the identity matrix never appears).

For a fixed $a \in A$, the two $b$'s that give us trace $\pm $2 matrices are $b_p$ and $b_m$, where
\[
b_p = \frac{2-a-a^{-1}}{i-j}, \quad \quad b_m = \frac{-2-a-a^{-1}}{i-j}.
\]

Denote the corresponding matrices by $M_p$, $M_m$.  I.e.,
\[
M_p = \begin{bmatrix} a-b_pj & b_p \\ ai-b_pij-a^{-1}j & b_pi+a^{-1} \end{bmatrix}, \quad \quad M_m = \begin{bmatrix} a-b_mj & b_m \\ ai-b_mij-a^{-1}j & b_mi+a^{-1} \end{bmatrix}.
\]
If $a = 1$, our claim is clear, so suppose $a \neq 1$.  In this case, neither $b_p$ nor $b_m$ are $0$.  So by Proposition 4.2, we have the following, where $\sim$ denotes conjugacy in $\PSL_2(\bb{F}_q)$.
\[
M_p \sim \begin{bmatrix} 1 & b_p \\ 0 & 1 \end{bmatrix}, \quad \quad M_m \sim\begin{bmatrix} 1 & -b_m \\ 0 & 1 \end{bmatrix}.
\]

Recall that the conjugacy classes of elements of $\PSL_2(\bb{F}_q)$ with $1$'s on the diagonal and a $0$ in the lower left corner are determined based on whether the element in the upper right corner is a zero, a square, or a non-square.  Since $\frac{b_p}{-b_m} =-\frac{(a-1)^2}{(a+1)^2}$, and $-1$ is not a square in $\bb{F}_q$, that means that $b_p$ and $-b_m$ are distinct in $\bb{F}_q^{\ast}/(\bb{F}_q^{\ast})^2$, and so $M_p$ and $M_m$ are not conjugate in $\PSL_2(\bb{F}_q)$.  

This is true for every value of $a \in A$, so adding up all of these terms for each value of $a$ gives the desired result when $q \equiv 3 \pmod{4}$.

Next, suppose $q \equiv 1 \pmod{4}$. One thing to note is that since $-1$ is a square in $\bb{F}_q$, $a$ is a square if and only if $-a$ is a square, and so we can talk about whether or not elements of $A$ are squares and this notion is independent of our choice of $A$. Again, for a fixed $a \in A$, the following sum contains an element of every trace:
\[
\sum_{b \in \bb{F}_q} \begin{bmatrix} a-bj & b \\ ai-bij-a^{-1}j & bi+a^{-1} \end{bmatrix}.
\]
When we project down to $\PSL_2(\bb{F}_q)$, we get 1 element in $\tr(0)$, 2 elements in $\tr(\alpha)$ for every $\alpha \in \bb{F}_q^{\ast}/\{\pm\}$ with $\alpha \neq \pm 2$, and 2 elements of trace $\pm 2$ (and again, since $i \neq j$, neither of these can be the identity).

Just as before, for a fixed $a \in A$, there are two $b$'s giving us matrices of trace $\pm 2$, namely
\[
b_p = \frac{2-a-a^{-1}}{i-j}, \quad \quad b_m = \frac{-2-a-a^{-1}}{i-j}.
\]
with $M_p$ and $M_m$ the corresponding matrices.

First, we will show $M_p$ is conjugate to $M_m$.  If $a = 1$, then using Proposition 4.2, we have the following:
\[
M_p \sim\begin{bmatrix}1 & -(i-j) \\ 0 & 1 \end{bmatrix}, \quad \quad M_m \sim \begin{bmatrix} 1 & \frac{4}{i-j} \\ 0 & 1 \end{bmatrix}.
\]
Since $q \equiv 1 \pmod{4}$, we have that $-(i-j)$ is a square if and only if $\frac{4}{i-j}$ is a square if and only if $(i-j)$ is a square. 

For $a \neq 1$, we again use proposition 4.2 to see
\[
M_p \sim \begin{bmatrix} 1 & b_p \\ 0 & 1 \end{bmatrix},  \quad \quad M_m \sim \begin{bmatrix} 1 & -b_m \\ 0 & 1 \end{bmatrix} .
\] 
Since $\frac{b_p}{-b_m} = -\frac{(a+1)^2}{(a-1)^2}$, we have that $b_p$ is a square if and only if $-b_m$ is a square, and so $M_p$ is conjugate to $M_m$.

Next, we will see when $M_p,M_m \in \tr{}^{\square}(2)$ and when $M_p,M_m \in \tr{}^{\not\square}(2)$. This amounts to determining when $-b_m = \frac{2+a+a^{-1}}{i-j}$ is a square.  If $a+a^{-1}+2 = \ell^2$ for some $\ell$, then $a = \Big(\frac{a+1}{\ell}\Big)^2$, and conversely if $a = k^2$, then $a+a^{-1}+2 = (k+k^{-1})^2$.  Therefore, $a+a^{-1}+2$ is a square if and only if $a$ is a square.  This gives us the following, which also holds when $a = 1$,
\begin{displaymath}
   M_p \sim M_m \in \left\{
     \begin{array}{lr}
       \tr^{\square}(2) & \text{ if } a \equiv i-j \text{ in } \bb{F}_q^{\ast}/(\bb{F}_q^{\ast})^2\\
       \tr^{\not\square}(2) & \text{ if }a \not\equiv i-j \text{ in } \bb{F}_q^{\ast}/(\bb{F}_q^{\ast})^2.
     \end{array}
   \right.
\end{displaymath} 

In either case, whether or not $i-j$ is a square, $\frac{q-1}{4}$ of the values of $a$ will give us 2 elements of $\tr^{\square}(2)$ and the remaining $\frac{q-1}{4}$ values of $a$ will give us 2 elements of $\tr^{\not\square}(2)$.

\end{proof}

Let us return to calculating $\star$.  From here on out, we will split it into two cases, based on $q \pmod{4}$.  The main difference (aside from the character tables) is that if $q \equiv 1 \pmod{4}$, then $B$ contains elements of trace 0, which doesn't happen if $q \equiv 3 \pmod{4}$. It is interesting to note that the final formula for $A_{\Phi}^0$ does not depend on whether $q \equiv 1 \text{ or } 3 \pmod{4}.$

First, suppose $q \equiv 3 \pmod{4}$.  Proposition 6.2 states the conjugacy classes appearing in $n_{-}(i)Bn_{-}(-j)$, but we still have the terms in $\star$ of the form $n_{-}(i)Bn_{-}(-i)$, which are conjugate to $B$.  
\begin{prop}
Suppose $q \equiv 3 \pmod{4}$.  Then, $B$ consists of the identity matrix, $\frac{q-1}{2}$ matrices in $\tr^{\square}(2)$, $\frac{q-1}{2}$ matrices in $\tr^{\not\square}(2)$, and $q$ matrices in $\tr(a+a^{-1})$ for every $a$ in $\bb{F}_q^{\ast}/\{\pm\}$.
\end{prop}  
Putting this all together gives the following formula. Recall that $A \subseteq \bb{F}_q^{\ast}$ is a set of representatives for $\bb{F}_q^{\ast}/\{\pm\}$ such that $1 \in A$. 
\begin{align*}
\star &=\frac{1-\rho}{q+1} \frac{1}{\#\PSL_2(\bb{F}_q)} \sum_{\sigma \in \PSL_2(\bb{F}_q)} \sigma \Bigg( \sum_{i,j \in \bb{F}_q} \epsilon_i\epsilon_j n_{-}(i)Bn_{-}(-j) \Bigg) \sigma^{-1} \\
&= \frac{1-\rho}{q+1} \Bigg[ \epsilon \Bigg( \operatorname{Id} + \frac{1}{q+1}(\tr^{\square}(2) + \tr^{\not\square}(2)) + \frac{1}{q+1}\sum_{x \in A\setminus\{1\}}\tr(x+x^{-1})  \Bigg) \\
&+ \epsilon(\epsilon-1) \Bigg(\frac{1}{q}\tr_{E^c/k} -\frac{1}{q}\operatorname{Id}-\frac{1}{q(q+1)}(\tr^{\square}(2) + \tr^{\not\square}(2)) - \frac{1}{q(q+1)}\sum_{x \in A\setminus\{1\}}\tr(x+x^{-1}) \Bigg)\Bigg]\\
&= \frac{\epsilon}{q+1}(1-\rho)\Bigg( \frac{q+1-\epsilon}{q(q+1)}\sum_{x \in A\setminus\{1\}}\tr(x+x^{-1}) + \frac{q+1-\epsilon}{q}\operatorname{Id} \\
&+ \frac{q+1-\epsilon}{q(q+1)}(\tr^{\square}(2) + \tr^{\not\square}(2)) + \frac{\epsilon-1}{q}\tr_{E^c/k} \Bigg)
\end{align*}

\begin{lem} Let $q \equiv 3 \pmod{4}$ and let $\chi_0$ denote the trivial character.  Then we have the following relations. Recall that $A \subseteq\bb{F}_q^{\ast}$ is a set of representatives for $\bb{F}_q^{\ast}/\{\pm\}$ such that $1 \in A$.
\begin{enumerate}[(a)]
\item $\displaystyle
\sum_{x \in A\setminus\{1\}} \tr(x+x^{-1}) =\chi_{\Ind_B^{\PSL_2(\bb{F}_q)}(\chi_0)} - \frac{2}{q-1}\chi_{\Ind_U^{\PSL_2(\bb{F}_q)}(\chi_0)}
$
\item $\displaystyle
\tr^{\square}(2) + \tr^{\not\square}(2) + (q+1)\operatorname{Id} = \frac{2}{q-1}\chi_{\Ind_U^{\PSL_2(\bb{F}_q)}(\chi_0)}
$
\end{enumerate}
\end{lem}
\begin{proof}
Using the orthogonality relations, we can find the inverse of the character table (thinking of the character table as a matrix).  This allows us to rewrite the indicator function of each conjugacy class in terms of the irreducible representations of $\PSL_2(\bb{F}_q)$.  From there, a straightforward calculation gives the result.
\end{proof}
Using lemma 6.4 to rewrite $\star$ gives 
\[
\star = \frac{\epsilon}{q+1}(1-\rho)\Bigg( \frac{q+1-\epsilon}{q(q+1)} \chi_{\Ind_B^{\PSL_2(\bb{F}_q)}(\chi_0)} + \frac{\epsilon-1}{q}\tr_{E^c/k} \Bigg)
\]

Combining the previous lemma and our work to calculate $\star$ gives the following formula for $A_{\Phi}^0$.
\begin{align*}
A_{\Phi}^0 &= \frac{1}{2}\tr_{E^c/k} -\frac{\epsilon(q+1-\epsilon)}{q(q+1)}(1-\rho)\tr_{E^c/k} + \frac{\epsilon(q+1-\epsilon)}{q(q+1)^2}(1-\rho)\chi_{\Ind_B^{\PSL_2(\bb{F}_q)}(\chi_0)}
\end{align*}

Now, suppose $q \equiv 1 \pmod{4}$.  Proposition 6.2 states the conjugacy classes appearing in $n_{-}(i)Bn_{-}(-j)$.  The remaining terms are of the form $n_{-}(i)Bn_{-}(-i)$, which are all conjugate to $B$.

\begin{prop}
Suppose $q \equiv 1 \pmod{4}$.  Then, $B$ consists of the identity matrix, $\frac{q-1}{2}$ matrices in $\tr^{\square}(2)$, $\frac{q-1}{2}$ matrices in $\tr^{\not\square}(2)$, $q$ matrices in $\tr(0)$, and $q$ matrices in $\tr(a+a^{-1})$ for every $a \in \bb{F}_q^{\ast}/\{\pm\}$.  
\end{prop}

Putting this all together gives the following formula for $\star$, 
\begin{align*}
\star &=\frac{1-\rho}{q+1} \frac{1}{\#\PSL_2(\bb{F}_q)} \sum_{\sigma \in \PSL_2(\bb{F}_q)} \sigma \Bigg( \sum_{i,j \in \bb{F}_q} \epsilon_i\epsilon_j n_{-}(i)Bn_{-}(-j) \Bigg) \sigma^{-1} \\
&= \frac{\epsilon}{q+1}(1-\rho)\Bigg[\operatorname{Id} + \frac{1}{q+1}(\tr^{\square}(2) + \tr^{\not\square}(2)) + \frac{2}{q+1}\tr(0)\\
 &+\frac{1}{q+1}\sum_{x \in A\setminus\{1, \sqrt{-1}\}}\tr(x+x^{-1}) + (\epsilon-1)\Bigg(\frac{1}{q}\tr_{E^c/k} -\frac{1}{q}\operatorname{Id}-\frac{2}{q(q+1)}\tr(0) \\
 &-\frac{1}{q(q+1)}(\tr^{\square}(2) + \tr^{\not\square}(2)) -\frac{1}{q(q+1)}\sum_{x \in A\setminus\{1,\sqrt{-1}\}} \tr(x+x^{-1}) \Bigg)\Bigg]
\end{align*}

We will use the following lemma to simplify the above sum, whose proof is a standard calculation using the inverted character table similar to Lemma 6.4.
\begin{lem}
 Let $q \equiv 1 \pmod{4}$ and let $\chi_0$ denote the trivial character.  Then we have the following relations. Recall $A \subseteq \bb{F}_q^{\ast}$ is a set of representatives for $\bb{F}_q^{\ast}/\{\pm\}$.
 \begin{enumerate}[(a)]
 \item $\displaystyle 2\tr(0)+ \sum_{x \in A\setminus\{1,\sqrt{-1}\}}\tr(x+x^{-1}) = \chi_{\Ind_B^{\PSL_2(\bb{F}_q)}(\chi_0)} - \frac{2}{q-1}\chi_{\Ind_U^{\PSL_2(\bb{F}_q)}(\chi_0)}
$
\item$\displaystyle
\tr^{\square}(2) + \tr^{\not\square}(2) + (q+1)\operatorname{Id} = \frac{2}{q-1}\chi_{\Ind_U^{\PSL_2(\bb{F}_q)}(\chi_0)}
$
 \end{enumerate}
\end{lem}

Using this lemma to rewrite $\star$ gives
\[
\star = \frac{\epsilon}{q+1}(1-\rho)\Bigg( \frac{q+1-\epsilon}{q(q+1)} \chi_{\Ind_B^{\PSL_2(\bb{F}_q)}(\chi_0)} + \frac{\epsilon-1}{q}\tr_{E^c/k} \Bigg).
\]

And again, we use our previous calculations to conclude the proof showing the following formula,
\begin{align*}
A_{\Phi}^0 &= \frac{1}{2}\tr_{E^c/k} -\frac{\epsilon(q+1-\epsilon)}{q(q+1)}(1-\rho)\tr_{E^c/k} + \frac{\epsilon(q+1-\epsilon)}{q(q+1)^2}(1-\rho)\chi_{\Ind_B^{\PSL_2(\bb{F}_q)}(\chi_0)}.
\end{align*}
\end{proof}

\section{Conclusions}
The Colmez conjecture states that $\hfal(\Phi) = -Z(0,A_{\Phi}^0)$ for a CM type $\Phi$.  Before applying $Z(0,\cdot)$ to $A_{\Phi}^0$, we will rewrite $A_{\Phi}^0$ in terms of simple characters induced from subgroups of $\Gal(E^c/\bb{Q})$.  This has the advantage of simplifying the resulting $L$-functions.

Let $\chi_{k/\bb{Q}}, \chi_{E/F}$ denote the quadratic characters associated to the extensions $k/\bb{Q}$ and $E/F$ respectively.  We view $\chi_{k/\bb{Q}}$ as a function on $\Gal(E^c/\bb{Q})$ via the quotient map $\Gal(E^c/\bb{Q}) \twoheadrightarrow \Gal(k/\bb{Q})$.  We view $\chi_{E/F}$ as a function on $\Gal(E^c/F)$ via the quotient map $\Gal(E^c/F) \twoheadrightarrow \Gal(E/F)$.  Using this notation, the following lemma is a standard calculation.

\begin{lem} As functions on $\Gal(E^c/\bb{Q})$ we have the following relations:
\begin{enumerate}[(a)]
\item \[\displaystyle \tr_{E^c/k} = \frac{1}{2}\chi_{\Ind_{\Gal(E^c/k)}^{\Gal(E^c/\bb{Q})}(\chi_0)}\]
\item \[\displaystyle (1-\rho)\tr_{E^c/k} = \chi_{k/\bb{Q}}\]
\item \[\displaystyle(1-\rho)\chi_{\Ind_B^{\PSL_2(\bb{F}_q)}(\chi_0)} = \chi_{\Ind_{\Gal(E^c/F)}^{\Gal(E^c/\bb{Q})} (\chi_{E/F})}\]
\end{enumerate}
\end{lem}



Using this lemma, we can rewrite $A_{\Phi}^0$ as
\begin{align*}
A_{\Phi}^0 &= \frac{1}{4}\chi_{\Ind_{\Gal(E^c/k)}^{\Gal(E^c/\bb{Q})}(\chi_0)} -\frac{\epsilon(q+1-\epsilon)}{q(q+1)}\chi_{k/\bb{Q}} + \frac{\epsilon(q+1-\epsilon)}{q(q+1)^2} \chi_{\Ind_{\Gal(E^c/F)}^{\Gal(E^c/\bb{Q})} (\chi_{E/F})}.
\end{align*}

Finally combining all of the above results gives the following theorem.
\begin{thm}
The Colmez Conjecture is true for $E$.  That is, let $\Phi$ be a CM type of $E$ of signature $(q+1-\epsilon,\epsilon)$.
Then,
\begin{align*}
\hfal(\Phi) &=-Z(0,A_{\Phi}^0) \\
&= -\frac{1}{4}Z(0,\zeta_k)+\frac{\epsilon(q+1-\epsilon)}{q(q+1)}Z(0,\chi_{k/\bb{Q}}) - \frac{\epsilon(q+1-\epsilon)}{q(q+1)^2}Z(0,\chi_{E/F}).
\end{align*}
\end{thm}
\begin{proof}
Let $\Phi(E)_{\epsilon}$ denote all CM types of signature $(q+1-\epsilon,\epsilon)$.  A result of Yang and Yin \cite{YangYin} states
\begin{equation}
\sum_{\Phi \in \Phi(E)_{\epsilon}} \hfal(\Phi) = \sum_{\Phi \in \Phi(E)_{\epsilon}} -Z(0,A_{\Phi}^0).
\end{equation}

Our theorem 6.1 shows that $A_{\Phi}^0$ is constant among CM types of a given signature.  Colmez's Th\'{e}or\`{e}me 0.3 \cite{Col} shows that the Faltings height of a CM type $\Phi$ only depends on $A_{\Phi}^0$.  If $\Phi$ is a CM type of $E$ of signature $(q+1-\epsilon,\epsilon)$, then equation 7.1 implies that $\hfal(\Phi) = -Z(0,A_{\Phi}^0)$.  The expression for $\hfal(\Phi)$ is due to our calculation of $A_{\Phi}^0$ immediately preceding this theorem as well as the linearity of $Z(0,\cdot)$.
\end{proof}

\section{Acknowledgements} The author thanks Tonghai Yang for all of his help. The author would also like to thank Alisha Zachariah and Shamgar Gurevich for their help and references regarding $\PSL_2(\bb{F}_q)$.  Finally, the author would like to thank Juliette Bruce for her valuable feedback on this paper. This work was done with the support of National Science Foundation grant DMS-1502553.

\bibliographystyle{alpha}
\bibliography{/Users/salvatoreparenti/Desktop/TexFiles/mybib}  

\begin{thebibliography}{{Yan}10b}

\bibitem[{Ada}]{Adams}
J.~{Adams}.
\newblock Character tables for {GL(2)}, {SL(2)}, {PGL(2)} and {PSL(2)} over a
  finite field.
\newblock Website.

\bibitem[AGHM15]{AGHM}
F.~{Andreatta}, E.~{Goren}, B.~{Howard}, and K.~{Madapusi Pera}.
\newblock {Faltings heights of abelian varieties with complex multiplication}.
\newblock {\em ArXiv e-prints}, August 2015.

\bibitem[BM16]{BSM}
A.~{Barquero-Sanchez} and R.~{Masri}.
\newblock On the colmez conjecture for non-abelian cm fields.
\newblock {\em Submitted}, 2016.

\bibitem[{Col}93]{Col}
P.~{Colmez}.
\newblock P{\'e}riodes des vari{\'e}t{\'e}s ab{\'e}liennes {\`a} multiplication
  complexe.
\newblock {\em Ann. of Math.}, 138:625--683, 1993.

\bibitem[{Mil}06]{MilneCM}
J.~{Milne}.
\newblock Complex multiplication.
\newblock www.jmilne.org/math, 2006.

\bibitem[{Obu}13]{Obus}
A.~{Obus}.
\newblock On colmez's product formula for periods of cm-abelian varieties.
\newblock {\em Math. Ann.}, 356(2):401--418, 2013.

\bibitem[{Yan}10a]{Yang10}
T.~{Yang}.
\newblock An arithmetic intersection formula on hilbert modular surfaces.
\newblock {\em Amer. J. Math}, 132:1275--1309, 2010.

\bibitem[{Yan}10b]{Yang}
T.~{Yang}.
\newblock The chowla-selberg formula and the colmez conjecture.
\newblock {\em Canad. J. Math}, 62:456--472, 2010.

\bibitem[{Yan}13]{Yang13}
T.~{Yang}.
\newblock Arithmetic intersection and faltings' height.
\newblock {\em Asian J. Math}, 17:335--382, 2013.

\bibitem[YY]{YangYin}
T.~{Yang} and H.~{Yin}.
\newblock Cm fields of dihedral type and colmez's conjecture.
\newblock In Preparation.

\bibitem[YZ15]{YZ}
X.~{Yuan} and S.-W {Zhang}.
\newblock On the averaged colmez conjecture.
\newblock {\em ArXiv e-prints}, July 2015.

\end{thebibliography}

\end{document}